\documentclass[12pt]{amsart}

\usepackage{amssymb,latexsym}
\usepackage{enumerate}
\usepackage[hyperfootnotes=false]{hyperref}
\usepackage{fullpage}
\usepackage{soul}

\usepackage{fixme}
\fxsetup{
    status=draft,
    author=,
    layout=margin,
    theme=color
}

\makeatletter \@namedef{subjclassname@2010}{%
  \textup{2010} Mathematics Subject Classification}
\makeatother

\newcounter{thm} \numberwithin{thm}{section}
\newtheorem{Theorem}[thm]{Theorem}

\newtheorem{Definition}[thm]{Definition}

\newtheorem{Claim}[thm]{Claim}

\newcommand{\thistheoremname}{}
\newtheorem*{genericthm*}{\thistheoremname}
\newenvironment{namedthm*}[1]
  {\renewcommand{\thistheoremname}{#1}%
   \begin{genericthm*}}
  {\end{genericthm*}}
\usepackage{tikz}
\usetikzlibrary{decorations.pathreplacing}
\tikzset{mybrace/.style={decoration={brace,raise=1.8mm},decorate}}
\tikzset{mybracedown/.style={decoration={brace,mirror,raise=1.8mm},decorate}}
\usetikzlibrary{math}
\usetikzlibrary{patterns}
\usetikzlibrary{matrix,backgrounds}

\date{}

\author[K. Bhowmick]{Krishnendu Bhowmick} \address{Johann Radon Institute for Computational and Applied Mathematics\\
Linz, Austria}
\email{Krishnendu.Bhowmick@oeaw.ac.at}

\author[M. Patry]{Miriam Patry} 
\email{m.patry@outlook.at}

\author[O. Roche-Newton]{Oliver Roche-Newton} \address{Institute for Algebra, Johannes Kepler Universit\"{a}t\\
Linz, Austria}
\email{o.rochenewton@gmail.com}

\begin{document}

\baselineskip=17pt

\title{Local differences determined by convex sets}

\begin{abstract} 

This paper introduces a new problem concerning additive properties of convex sets. Let $S= \{s_1 < \dots <s_n \}$ be a set of real numbers and let $D_i(S)= \{s_x-s_y: 1 \leq x-y \leq i\}$. We expect that $D_i(S)$ is large, with respect to the size of $S$ and the parameter $i$, for any convex set $S$. 

We give a construction to show that $D_3(S)$ can be as small as $n+2$, and show that this is the smallest possible size. On the other hand, we use an elementary argument to prove a non-trivial lower bound for $D_4(S)$, namely $|D_4(S)| \geq \frac{5}{4}n -1$. For sufficiently large values of $i$, we are able to prove a non-trivial bound that grows with $i$ using incidence geometry.
\end{abstract}

\date{}
\maketitle

\section{Introduction} 
Let \(S\subseteq \mathbb{R}\) be a finite set and write \(S = \{s_1,s_2,\dots,s_n\}\) such that \(s_i<s_{i+1}\)  for all \(i\in[n-1]\). $S$ is said to be a \textit{convex set} if \(s_{i+1}-s_{i} < s_{i+2}-s_{i+1}\) for all \(i\in [n-2]\).

In the spirit of the sum-product problem, Erd\H{o}s raised the question of whether $A+A$ and $A-A$ are guaranteed to be large for any convex set. Progress has been made towards this question via a combination of methods in incidence geometry (see for instance \cite{ENR}) and elementary methods (see \cite{G}, \cite{RSSS} and \cite{HRNR}), but the question of determining the best possible lower bounds remains open. The current state-of-the-art results state that, for some absolute constant $c>0$, the bounds
\begin{equation} \label{known}
|S-S| \gg \frac{ n^{8/5}}{(\log n)^c}, \,\,\,\,\, |S+S| \gg \frac{  n^{30/19}}{(\log n)^c}
\end{equation}
hold for any convex set $A \subset \mathbb R$, see \cite{SS} and \cite{RS} respectively. The notation $\gg$ is used to absorb a multiplicative constant. That is $X \gg Y$ denotes that there exists an absolute constant $C>0$ such that $X \geq CY$.

In this paper, we introduce a new \textit{local} variant of this problem, where only differences between close elements of $A$ are considered.
We define 
\[
D_i(S):=\{s_x - s_{y}: 1 \leq x-y \leq i \}.
\]
Moreover, for any set \(\mathcal{I}\subset [n-1]\), define 
\[
D_{\mathcal{I}}(S):=\{s_x - s_{y}\mid x-y \in \mathcal{I} \}.
\]
We expect that $D_i(S)$ is large (in terms of $i$) and that $D_{\mathcal{I}}(S)$ is large (in terms of $| \mathcal I|$) for any convex set $S \subset \mathbb R$. We begin with some trivial observations.
\begin{itemize}
    \item For any convex set $S$, $|D_1(S)|=|S|-1$. Indeed, it is an immediate consequence of the definition of a convex set that $S$ has \textit{distinct consecutive differences}. Additive properties of such sets were considered in \cite{RSSS} and  \cite{RuSol}.
    \item For any convex set $S$, we have $|D_2(S)| \geq |S|$. This can be seen by observing that $s_n - s_{n-2} \in D_{ \{2\}}(S) \setminus D_{ \{1\}}(S)$, and so
    \begin{equation} \label{lowerbound2}
    |D_2(S)|= | D_{ \{1\}}(S)| + | D_{ \{2\}}(S) \setminus D_{ \{1\}}(S)| \geq n-1 + 1 = n.
    \end{equation}
    On the other hand, if we consider the set
    \begin{equation} \label{eq 1}
 S = \{f_{i+3} : i \in [n]\},
    \end{equation}
    where  \(f_i\) is the \(i\)-th Fibonacci number, we see that \(|D_2(S)| = n\), and so the lower bound \eqref{lowerbound2} is optimal.
    \item Setting $i=n-1$, we see that the set $D_i(S)$ is equal to the set of positive elements of the difference set $S-S$. Therefore, the question of obtaining lower bounds for $D_{n-1}(S)$ is equivalent to that of lower bounding the size of $S-S$.
\end{itemize}

The main results of this paper concern the minimum possible sizes of $D_3(S)$ and $D_4(S)$ when $S$ is convex. Let \(\mathcal{S}(n)\) be the set of all convex sets of size \(n\). Define 
\[
D_i(n) := \min\{|D_i(S)|: S\in \mathcal{S}(n)\}.
\]
The discussion above shows that \(D_1(n) = n-1\) and \(D_2(n) = n\). For \(D_3(n)\) we will provide a constructive proof of the following result.
 \begin{Theorem}\label{Thm 1} For \(n\geq 5\)
 \[
 D_3(n) = n+2.
 \]
 \end{Theorem}
However, there is an interesting change of behaviour that occurs when $i$ goes from $3$ to $4$, and \(D_4(n)\) is not of the form \(n + O(1)\). In this paper we will prove the following result.
\begin{Theorem}\label{Thm 2}
  \[ 
  D_4(n) \geq \frac{5}{4}n - 1.
  \]  
\end{Theorem}
We also consider a stronger notion of convexity, namely \(k\)-convexity, which was one of the main considerations of the paper \cite{HRNR}.
 \begin{Definition}
\(S\) is called a \(k\)-convex set if \(D_1(S)\) is a \((k-1)\)-convex set. A convex set is also called  a \(1\)-convex set.
 \end{Definition}
Let \(\mathcal{S}^2(n)\) be the set of all \(2\)-convex sets of size \(n\). Define 
\[
D_i^2(n) := \min\{|D_i(S)|: S\in \mathcal{S}^2(n)\}.
\]
Note that \(D_1^2(n) = n-1\). The set $S$ defined in (\ref{eq 1}) is a \(2\)-convex set with \(|D_2(S)| = n\), and it therefore follows that \(D_2^2(n) = n\).

Similarly to Theorem \ref{Thm 1} we will prove the following result.
 \begin{Theorem}\label{Thm 2convex} For \(n\geq 5\)
 \begin{equation} \label{lowerbound}
 D_3^2(n) = n+2.
 \end{equation}
 \end{Theorem}
We give a slightly better lower bound for \(D_4^2(n)\) than that for \(D_4(n)\) given in Theorem \ref{Thm 2}.
 \begin{Theorem}\label{Thm 3}
 \[D_4^2(n) \geq \frac{4}{3}n - \frac{4}{3}.\]
 \end{Theorem}

We expect that the values of $D_i(n)$ and $D_i^2(n)$ increase with $i$. However, we are not able to generalise the elementary techniques used to prove Theorems \ref{Thm 2} and \ref{Thm 2convex} to obtain better bounds by considering larger $i$. On the other hand, for sufficiently large $i$, incidence geometric techniques can be used to give a non-trivial bound.

\begin{Theorem} \label{thm:largei} For all $i \in \mathbb N$,
\[
D_i(n) \gg i^{3/2}.
\]
\end{Theorem}

In particular, if $i \geq  n^{2/3+\epsilon}$, then Theorem \ref{thm:largei} gives a lower bound of the form $D_i(n) \gg n^{1+\epsilon'}$. It appears plausible that such a lower bound for $D_i(n)$ holds under the weaker assumption that $i \geq n^{\epsilon}$.


\section{Proofs of the main results}

 \begin{proof}[Proof of Theorems \ref{Thm 1} and \ref{Thm 2convex}]

 To prove that $D_3(n) \geq n+2$, we need to prove that
 \begin{equation} \label{lowerboundd}
 |D_3(S)| \geq n+2
 \end{equation}
 holds for an arbitrary convex set $S$ with $|S|=n$. Observe that
 \(|D_{3}(S)| \geq |D_{2}(S)| + 1\), since \(s_n-s_{n-3}\in D_3(S)  \setminus D_{2}(S)\). Hence \(|D_3(S)| \geq n+1\). Suppose for a contradiction that \(|D_3(S)| = n+1\) for some \(n \geq 5\).  Observe that 
 \[
  s_2-s_1 < s_3-s_2 < s_3-s_1 < s_4-s_2 < s_4-s_1 < s_5-s_2 < \dots < s_n-s_{n-3}
 \] 
 and 
 \[
  s_2-s_1 < s_3-s_2 < s_4-s_3 < s_4-s_2 < s_5-s_3 < s_5-s_2 < \dots < s_n-s_{n-3}.
 \] 
 We have identified two increasing sequences of length $n+1$ in $D_3(S)$, and so the sequences must be identical. Comparing the third terms of the sequences, we get
 \begin{equation}\label{eq 2}
     s_4-s_3 = s_3-s_1
 \end{equation}
 and comparing the fifth terms of the sequences, we get
 \begin{equation}\label{eq 3}
     s_4-s_1 = s_5-s_3.
\end{equation}
From equations \eqref{eq 2} and \eqref{eq 3} we get 
\[
s_5-s_4 = s_3 - s_1 = s_4- s_3.
\] 
But as \(S\) is a convex set this is not possible. Hence our assumption that \(|D_3(S)| = n+1\) was wrong, and so \(|D_3(S)| \geq n+2\), proving \eqref{lowerboundd}.

Since every $2$-convex set is also convex, it follows that $|D_3(S)| \geq n+2$ for all $2$-convex sets $S$. Therefore, $D_3^2(n) \geq n+2$.

 To prove that $D_3(n) \leq n+2$ and $D_3^2(n) \leq n+2$, we provide an example of a $2$-convex (and thus also convex) set  with \(|D_3(S)|=n+2\). Consider the set \(S = \{s_1,s_2,\dots,s_n\}\) where 
 \begin{align*}
& s_1 = 0, \, s_2 = 10, \, s_3 = 23, \, s_4 = 40,
 \\ & s_i:= a_{i-1}+s_{i-2}-s_{i-4} \,\,\,\,  \text{for }  i=5,\dots,n.
 \end{align*} Indeed, \(|D_3(S)| = n+2\) and $S$ is $2$-convex.

 \end{proof}

Taking a closer look at the construction of the set $S$ defined above, we see that there is additional structure, as the set $D_{\{5 \}}(S)$ overlaps significantly with $D_3(S)$. To be precise, we have
\[
|D_{\{1,2,3,5\}}(S)|=n+4.
\]
Since the two largest elements of $D_{\{1,2,3,5\}}(S)$ are not in $D_{3}(S)$, it follows from Theorem \ref{Thm 1} that $D_{\{1,2,3,5\}}(S) \geq n+4$ for any convex set $S$, and so this construction is optimal in this regard.

\begin{proof}[Proof of Theorem \ref{Thm 2}]

To prove Theorem \ref{Thm 2}, we need to show that
\begin{equation} \label{lowerbound3}
|D_4(S)| \geq \frac{5}{4}n - 1
\end{equation}
holds for an arbitrary convex set $S$ with size $n$. Indeed, let $S$ be such a set. The proof of \eqref{lowerbound3} is split into two cases.

\textbf{Case 1} - Suppose that $|D_{\{2\}}(S) \setminus D_{\{1\}}(S)| \geq \frac{n}{4} -2 $. It follows that
\[
|D_4(S)| \geq |D_2(S)| +2 = |D_{ \{1 \}}(S) \cup [D_{\{2\}}(S) \setminus D_{\{1\}}(S)]| +2 \geq (n-1) +(n/4-2)+2=\frac{5n}{4}-1.
\]

\textbf{Case 2} - Suppose that $|D_{\{2\}}(S) \setminus D_{\{1\}}(S)| < n/4 -2 $. 
It follows that
\[
s_{i+2}-s_i ,s_{i+3}-s_{i+1} \in D_{ \{ 1 \}}(S)
\]
holds for at least 
\[
(n-3)- 2(n/4-2) = \frac{n}{2} +1
\]
values of $i \in [n-3]$. Let $I$ be the set of all such $i$. Then, for all $i \in I$, 
\[
s_{i+2}-s_i=s_{j+1}-s_j, \,\,\,\, s_{i+3}-s_{i+1} =s_{j'+1}-s_j,
\]
for some $j,j' \in [n-1]$.

\textbf{Case 2a} - Suppose that, for at least $\frac{n}{4} +2 $ of the elements $i \in I$, we have $j'=j+1$. Then $s_{j+2} - s_j$ is in $D_{\{2\}}(S)$, but it is also strictly between two consecutive elements of \(D_{\{4\}}(S)\). Indeed
 \begin{equation*}
    \begin{split}
    s_{i+3}-s_{i-1}  & = (s_{i+3}-s_{i+1}) + (s_{i+1}-s_{i-1})\\
    & < (s_{i+3}-s_{i+1}) + (s_{i+2}-s_{i})\\ 
    & = (s_{j+2}-s_{j+1}) + (s_{j+1}-s_{j})\\
    & = s_{j+2}-s_{j} \\
    & < (s_{i+4}-s_{i+2}) + (s_{i+2}-s_{i})\\ 
    & = s_{i+4}-s_{i}.
\end{split}
\end{equation*} 
It follows that $|D_{ \{ 2 \}}(S) \setminus D_{\{4\}}(S)| \geq \frac{n}{4}+2$. Therefore,
\[
|D_4(S)| \geq 1 + |D_{\{2,4 \}}(S)| = 1+ |D_{\{4\}}(S) \cup [ D_{ \{ 2 \}}(S) \setminus D_{\{4\}}(S)]| \geq 1 + n-4 +\frac{n}{4} +2 = \frac{5n}{4}-1.
\]

\textbf{Case 2b} - Suppose that we are not in case 2a, and hence, for at least $\frac{n}{4}-1$ of the elements of $I$, we have $j' \geq j+2$. Then $s_{j+2}-s_{j+1}$ lies strictly between two consecutive elements of $D_{ \{2 \}}$, namely $s_{i+2}-s_i$ and $s_{i+3}-s_{i+1}$. Therefore, $|D_{ \{ 1 \}}(S) \setminus D_{\{2\}}(S)| \geq \frac{n}{4}-1$, and thus
\[
|D_4(S)| \geq  |D_2(S)| +2 = |D_{ \{2 \}}(S) \cup [D_{\{1\}}(S) \setminus D_{\{2\}}(S)]| +2 \geq (n-2)+  (n/4-1) + 2= \frac{5n}{4}-1.
\]

\end{proof}

A closer look at the proof of Theorem \ref{Thm 2} reveals that we have barely used the elements of $D_{ \{  3 \}}(S)$ anywhere in the proof of \eqref{lowerbound3}. By modifying the proof slightly, we obtain the bound
\[
|D_{ \{1,2,4\}}(S)| \geq \frac{5}{4}n - 2
\]
for any convex set $S$ with cardinality $n$.

The construction of the set $S$ from Theorem \ref{Thm 1} yields the bound
\[
|D_4(S)| \leq |D_3(S)| + (n-4) = 2n-2.
\]
Combining this observation with the result of Theorem \ref{Thm 2}, we see that
\[
\frac{5}{4}n -1 \leq D_4(n) \leq 2n-2.
\]

We now proceed to the proof of Theorem \ref{Thm 3}. The proof is largely the same as that of Theorem \ref{Thm 2}. The main difference is that we can use the $2$-convex condition to show that Case 2b does not occur.

\begin{proof}[Proof of Theorem \ref{Thm 3}]

To prove Theorem \ref{Thm 3}, we need to show that
\begin{equation} \label{lowerbound4}
|D_4(S)| \geq \frac{4}{3}n - \frac{4}{3}
\end{equation}
holds for an arbitrary $2$-convex set $S$ with size $n$. Indeed, let $S$ be such a set. The proof is split into two cases.

\textbf{Case 1} - Suppose that $|D_{\{2\}}(S) \setminus D_{\{1\}}(S)| \geq \frac{n}{3} - \frac{7}{3} $. It follows that
\[
|D_4(S)| \geq |D_2(S)| +2 = |D_{ \{1 \}}(S) \cup [D_{\{2\}}(S) \setminus D_{\{1\}}(S)]| +2 \geq (n-1) +\left(\frac{n}{3} - \frac{7}{3}\right)+2=\frac{4}{3}n-\frac{4}{3}.
\]

\textbf{Case 2} - Suppose that $|D_{\{2\}}(S) \setminus D_{\{1\}}(S)| < \frac{n}{3} - \frac{7}{3} $. 
It follows that
\[
s_{i+2}-s_i ,s_{i+3}-s_{i+1} \in D_{ \{ 1 \}}(S)
\]
holds for at least 
\[
(n-3)- 2\left (\frac{n}{3}-\frac{7}{3} \right) = \frac{n}{3} + \frac{5}{3}
\]
values of $i \in [n-3]$. Let $I$ be the set of all such $i$. Then, for all $i \in I$, 
\[
s_{i+2}-s_i=s_{j+1}-s_j, \,\,\,\, s_{i+3}-s_{i+1} =s_{j'+1}-s_j,
\]
for some $j,j' \in [n-1]$ satisfying $j' > j >i$. We claim now that it must be the case that $j'=j+1$. Indeed, suppose for a contradiction that $j' \geq j+2$. It then follows that
\begin{align*}
(s_{j+3}-s_{j+2}) - (s_{j+1}-s_j) &\leq (s_{j'+1}-s_{j'}) - (s_{j+1}-s_j) 
\\&= (s_{i+3}-s_{i+1})-(s_{i+2}-s_{i})
\\& = (s_{i+3}-s_{i+2})-(s_{i+1}-s_{i}) .
\end{align*}
However, since $j >i$, this contradicts the assumption that $S$ is $2$-convex. Indeed, write the convex set $D_1(S)= \{d_1 < d_2 < \dots < d_{n-1}\}$, and so $d_i=s_{i+1}-s_i$. Then the previous inequality can be written as
\[
d_{j+2} - d_j \leq d_{i+2} - d_i.
\]
But this inequality cannot hold if $D_1(S)$ is convex, which proves the claim.

As was the case in the proof of Theorem \ref{Thm 2}, $s_{j+2} - s_j$ is in $D_{\{2\}}(S)$, but it is also strictly in-between two consecutive elements of \(D_{\{4\}}(S)\). Indeed
 \begin{equation*}
    \begin{split}
    s_{i+3}-s_{i-1}  & = (s_{i+3}-s_{i+1}) + (s_{i+1}-s_{i-1})\\
    & < (s_{i+3}-s_{i+1}) + (s_{i+2}-s_{i})\\ 
    & = (s_{j+2}-s_{j+1}) + (s_{j+1}-s_{j})\\
    & = s_{j+2}-s_{j} \\
    & < (s_{i+4}-s_{i+2}) + (s_{i+2}-s_{i})\\ 
    & = s_{i+4}-s_{i}.
\end{split}
\end{equation*} 
It follows that $|D_{ \{ 2 \}}(S) \setminus D_{\{4\}}(S)| \geq \frac{n}{3}+\frac{5}{3}$. Therefore,
\[
|D_4(S)| \geq 1 + |D_{\{2,4 \}}(S)| = 1+ |D_{\{4\}}(S) \cup [ D_{ \{ 2 \}}(S) \setminus D_{\{4\}}(S)]| \geq 1 + n-4 +\frac{n}{3} + \frac{5}{3} = \frac{4}{3}n - \frac{4}{3}.
\]

\end{proof}

We now turn to the proof of Theorem \ref{thm:largei}, which is a modification of the proof of the main result in \cite{ENR}.

\begin{proof}[Proof of Theorem \ref{thm:largei}]

Let $S = \{ s_1 <  \dots <s_n \}$ be a convex set. We will prove that
\[
D_i(S) \gg i^{3/2}.
\]
Since $S$ is convex, it follows that there exists a strictly convex function $f: \mathbb R \rightarrow  \mathbb R$ such that $f(i)=s_i$.


Define $P$ to be the point set
\[
P=\{ -n, -n+1, \dots, n-1,n \} \times D_i(S).
\]
Let $\ell_{a,b}$ denote the curve with equation $y= f(x+a)-b$. Define $L$ to be the set of curves
\[
L= \{ \ell_{j,s_h} : 1 \leq j \leq n, 1 \leq h \leq n/2 \}.
\]
The set $L$ consists of translates of the same convex curve, and it therefore follows that we have the Szemer\'{e}di-Trotter type bound
\begin{equation} \label{ST}
I(P,L) \ll |P|^{2/3}|L|^{2/3} + |P| + |L|,
\end{equation}
where
\[
I(P,L):= \{ (p, \ell) \in P \times L : p \in \ell \}.
\]
For each $\ell_{j,s_h} \in L$, and for any $k$ such that $1 \leq k-h \leq i$, observe that
\[
(k-j, s_k-s_h) \in P \cap \ell_{j,s_h}.
\]
This implies that $I(P,L) \gg |L|i \gg n^2i$. Comparing this bound with \eqref{ST} yields
\[
n^2i \ll n^2 |D_i(S)|^{2/3} + n|D_i(S)| + n^2.
\]
If the third term on the right hand side is dominant then $i=O(1)$ and Theorem \ref{thm:largei} holds trivially. If the second term is dominant then we have
\[
|D_i(S)| \gg ni \geq i^{3/2}.
\]
If the third term is dominant then we obtain the stated bound
\[
|D_i(S)| \gg i^{3/2}.
\]

\end{proof}

By making a small modification to the proof of Theorem \ref{thm:largei}, one can prove that, for any $G \subset [n] \times [n]$,
\begin{equation} \label{graphs}
| \{ s_x - s_y : (x,y) \in G \}| \gg \left (\frac{|G|}{n} \right ) ^{3/2}.
\end{equation}
Similar sum-product type results for restricted pairs have been considered in, for instance, \cite{ARS} and \cite{RN}. Note that the bound in \eqref{graphs} becomes meaningful when $|G|$ is significantly larger than $n$. A construction of a convex set with a rich difference in \cite{RNW} shows that the set $ \{ s_x - s_y : (x,y) \in G \}$ can have cardinality as small as one when $G$ has cardinality as large as $ n/2$.

\section{Concluding remarks}

\subsection{Sums instead of differences}

One may also consider a version of this question with sums. Let $E_i(S)$ denote the set
\[
E_i(S)= \{ s_x+s_y : 1 \leq x-y \leq i \}.
\]
However, it turns out that this modification to the question makes it rather straightforward to prove a non-trivial bound that grows with $i$. If we split $S$ into disjoint consecutive blocks $S_1 \cup S_2 \cup \dots \cup S_t$, with each block having $i$ elements (we possibly discard some elements to ensure that all blocks have exactly the same size, and so $t= \lfloor n/i \rfloor$), then the sum sets $S_j+S_j$ are pairwise disjoint. It then follows from the lower bound for $S_j+S_j$ given in \eqref{known} that
\[
|S+S| \geq \sum_{j=1}^t |S_j+S_j| \gg\frac{n}{i} \cdot \frac{ i^{30/19}}{ (\log n)^c} = \frac{ni^{11/19}}{ (\log n)^c}.
\]

\subsection{Sets $\mathcal I$ such that $D_{ \mathcal I}=n+O(1)$}

In Theorem \ref{Thm 1}, we have seen that $D_{\{1,2,3 \}}(n)=n+2$. There are other examples of sets $\mathcal I \subset [n-1]$ with cardinality 3 such that $D_{ \mathcal I}(n)=n+O(1)$. For example, we can define $S$ using the recurrence relation
\[
s_n= s_{n-2}+s_{n-3}-s_{n-6}.
\]
By choosing the initial elements of $S$ suitably, the set $S$ is convex. However, this recurrence relations gives rise to the system of equations
\[
s_j-s_{j-2}=s_{j-3}-s_{j-6} = s_{j-5} - s_{j-9}, \,\,\,\, \forall \,\, 10 \leq j \leq n.
\]
This implies that the elements of $D_ {\{2 \}}(S), D_ {\{3 \}}(S)$ and $D_ {\{4 \}}(S)$ are largely the same, with just a few exceptions occurring at the extremes of the three sets. It follows that $D_ {\{2 ,3,4\}}(S)=n+C$, for some absolute constant $C$. It appears likely that the same argument can be used to show that $D_{\{k,k+d,k+2d\}}(n)=n +O_{k,d}(1)$.

On the other hand, we have seen in this paper that $D_{ \{1,2,4 \}}(n) = \frac{5n}{4} -O(1)$. This raises the following question: can we classify the sets $\mathcal I$ with cardinality $3$ with the property that $D_{ \mathcal I}(n) =n+C$, where $C$ is some constant (which may depend on the elements of $\mathcal I$)?

\section*{Acknowledgements}

The authors were supported by the Austrian Science Fund FWF Project P 34180. Part of this work was carried out while the second author was doing an internship supported by FFG Project 895224 - JKU Young Scientists. We are grateful to Audie Warren and Dmitrii Zhelezov for helpful discussions.

\end{document}